\documentclass[12pt]{amsart}
\usepackage{amsmath,amssymb,amsthm}
\newtheorem{theorem}{Theorem}[section]
\newtheorem{lema}[theorem]{Lemma}
\newtheorem{cor}[theorem]{Corollary}
\newtheorem{prop}[theorem]{Proposition}
\theoremstyle{definition}
\newtheorem{defi}[theorem]{Definition}
\newtheorem{remark}[theorem]{Remark}

\usepackage{graphicx}

\numberwithin{equation}{section}
\newcommand{\R}{\mathbb{R}}

\newcommand{\Z}{\ensuremath{\mathbb{Z}}}
\begin{document}

\title[The BNS invariants of the generalized solvable Baumslag-Solitar groups]{The BNS invariants of the generalized solvable Baumslag-Solitar groups and of their finite index subgroups}
\author{Wagner Sgobbi}
\address{UFSCar, S\~{a}o Carlos SP, Brasil}
\email{wagnersgobbi@dm.ufscar.br}
\author{Peter Wong}
\address{Bates College, Department of Mathematics, Lewiston ME 04240, USA}
\email{pwong@bates.edu}

\date{\today}

\keywords{$Sigma$ invariants, $R_{\infty}$, generalized solvable Baumslag-Solitar groups}
\subjclass[2020]{Primary: 20F65; Secondary: 20E45}

\begin{abstract}
We compute the Bieri-Neumann-Strebel invariants $\Sigma^1$ for the generalized solvable Baumslag-Solitar groups $\Gamma_n$ and their finite index subgroups. Using $\Sigma^1$, we show that certain finite index subgroups of $\Gamma_n$ cannot be isomorphic to $\Gamma_{k}$ for any $k$. In addition, we use the BNS-invariants to give a new proof of property $R_\infty$ for the groups $\Gamma_n$ and their finite index subgroups.
\end{abstract}

\maketitle

\section{Introduction}

The Bieri-Neumann-Strebel invariant $\Sigma^1(G)$ \cite{BNS} of a finitely generated group $G$ is an important object of study in geometric group theory and has many connections to other areas of mathematics, especially with the Thurston norm in low dimensional topology. However, the computation of $\Sigma^1$ is very difficult in general and there are only few classes of groups for which $\Sigma^1$ is known (see e.g. \cite{KL} and the references therein). 

A group $G$ is said to have property $R_\infty$ if $R(\varphi)$ is infinite for every automorphism $\varphi \in Aut(G)$. Here, $R(\varphi)$ is the number of twisted conjugacy classes of $\varphi$, that is, the number of equivalence classes in $G$ given by the relation $g \sim h \Leftrightarrow zg\varphi(z)^{-1}=h$ for some $z \in G$.
Twisted conjugacy classes are important in topological fixed point theory. 

Let $X$ be a space with universal covering $\tilde{X}$ and $f:X \to X$ be a homeomorphism with induced automorphism $f_*:\pi_1(X) \to \pi_1(X)$. Then $R(f_*)$ is actually the number of (topological) lifting classes of $f$ in $\tilde{X}$ given by a deck transformation conjugation, which also partitions the fixed points of $f$ in $X$. This number is an upper bound for the Nielsen number $N(f)$, which is a sharp lower bound for the minimal number of fixed points in the homotopy class $[f]$ and one of the main objects of study in Nielsen Theory (see \cite{Jiang}). For instance in \cite{DaciWong}, property $R_\infty$ was used to show that for any $n \geq 5$, there exists a $n$-dimensional nilmanifold $M$ such that every self-homeomorphism $f:M \to M$ is isotopic to be fixed point free. 

The motivation for this work is \cite{TabackWongGamman} in which J. Taback and P. Wong showed property $R_\infty$ for the generalized solvable Baumslag-Solitar groups $\Gamma_n$ and for every group quasi-isometric to $\Gamma_n$, using geometric group theoretic techniques. In \cite{DaciDess}, D. Gon\c{c}alves and D. Kochloukova used the Bieri-Neumann-Strebel (BNS or $\Sigma^1$) invariant to deduce property $R_\infty$ for certain classes of groups, including a new proof of the property $R_{\infty}$ for the Thompson's group $F$. Since the $\Sigma$-invariants of the Baumslag-Solitar groups $BS(1,n)$ are sufficient to guarantee property $R_\infty$, it is natural to ask whether property $R_\infty$ for $\Gamma_n$ and for their finite index subgroups can also be deduced using $\Sigma^1$.

In this paper, we show that the property $R_{\infty}$ for $\Gamma_n$ and for their finite index subgroups can be deduced from their respective BNS-invariants. Here we compute the 
$\Sigma^1$ invariants of $\Gamma_n$ and of all its finite index subgroups $H$. We show that these invariants lie in an open hemisphere of the corresponding character spheres so that property $R_{\infty}$ follows from \cite{DaciDess}. Furthermore, we extend the result to any finite direct product of these groups. Using $\Sigma^1$, we show that there exist finite index sugbroups of $\Gamma_n$ that cannot be isomorphic to any $\Gamma_{k}$, in contrast to the fact that every finite index subgroup of a solvable Baumslag-Solitar group $BS(1,n)$ is again a $BS(1,k)$.

The paper is organized as follows. In section $2$ we compute the $\Sigma^1$ for $\Gamma_n$ (Theorem \ref{meusigmagamman}). In section $3$, we classify all the finite index subgroups $H$ of $\Gamma_n$ in terms of specific generators and index (Theorem \ref{fisgamman}), and give a presentation of $H$ (Theorem \ref{presforh}). Then we compute their $\Sigma^1$ invariant (Theorem \ref{meusigmah}) and use it to show that some $H$ cannot be a generalized solvable Baumslag-Solitar group (Theorem \ref{hnotgamman}). In section $4$ we use geometric arguments about the behavior of the induced homeomorphisms $\varphi^*:S(G) \to S(G)$ to show that finding some special invariant convex polytopes in the character sphere of a finitely generated group $G$ is sufficient to guarantee property $R_\infty$ for $G$. In section $5$, we give new proofs (Theorems \ref{meurinftygamman} and \ref{rinftyforh}) of property $R_\infty$ for the groups $\Gamma_n$ and $H$ above and also for any finite direct product of them (Theorem \ref{meurinfproddir}). Finally, in Proposition \ref{example}, we exhibit a family of groups $G$ where Theorem \ref{meuteopolytope} can be used to guarantee property $R_{\infty}$ without complete information on $\Sigma^1(G)$.

\section*{Acknowledgements} This paper is part of the first author's Ph.D. project, under the supervisions of Prof. Daniel Vendr\'uscolo (UFSCar - Brazil) and the second author, Prof. Peter Wong. The first author wants to thank both supervisors for their guidance, Bates College (Lewiston-ME, USA) for the acceptance of the project and Funda\c{c}\~ao de Amparo \`a Pesquisa do Estado de S\~ao Paulo (FAPESP) for the financial support during the research through processes 2017/21208-0 and 2019/03150-0. We thank Prof. D. Kochloukova for pointing out the earlier work \cite{BS} of Bieri-Strebel which simplifies the proof of Theorem \ref{meusigmagamman}.

\section{Computation of $\Sigma^1(\Gamma_n)$}

In this section we compute the $\Sigma^1$ invariants of the generalized solvable Baumslag-Solitar groups $\Gamma_n$.
First we recall the definition of the BNS-invariant $\Sigma^1(G)$ of a finitely generated group $G$. There are other equivalent definitions (see \cite{BNS} and \cite{Strebel}) but we employ the following for our purposes.

\begin{defi}
Let $G$ be a finitely generated group. The character sphere of $G$ is the quotient space
\[
S(G)=(Hom(G,\mathbb{R})-\{0\})/\sim \ = \{[\chi]\ |\ \chi \in Hom(G,\mathbb{R})-\{0\}\},
\] where $\chi \sim \chi' \Leftrightarrow r\chi=\chi'$ for some $r>0$.
\end{defi}

It is well known that if the free rank of the abelianized group $G^{ab}$ is $n$ with generators $x_1,...,x_n$, then $S(G) \simeq S^{n-1}$ with homeomorphism
\begin{align*}
  &\mathfrak h:S(G)\longrightarrow S^{n-1} \\
  &[\chi]\longmapsto \frac{(\chi(x_1),...,\chi(x_n))}{\Vert (\chi(x_1),...,\chi(x_n)) \Vert}.\\
\end{align*}

Following \cite{Strebel}, we have

\begin{defi}
Let $G$ be a finitely generated group with finite generating set $X \subset G$. Denote by $\Gamma=\Gamma(G,X)$ the Cayley graph of $G$ with respect to $X$. The first $\Sigma$-invariant (or $BNS$ invariant) of $G$ is
\[
\Sigma^1(G)=\{[\chi] \in S(G)\ |\ \Gamma_\chi\ \text{is connected}\},
\] where $\Gamma_\chi$ is the subgraph of $\Gamma$ whose vertices are the elements $g \in G$ with $\chi(g) \geq 0$ and whose edges are those of $\Gamma$ which connect two such vertices.
\end{defi}

The solvable Baumslag-Solitar group $BS(1,n), n>1$ is defined by the presentation
$$
BS(1,n)=\langle a,t \mid tat^{-1}=a^n \rangle.
$$

We consider the following solvable generalization of $BS(1,n)$.

\begin{defi}\label{defgamman}
Let $n\geq 2$ be a positive integer with prime decomposition $n={p_1}^{y_1}...{p_r}^{y_r}$, the $p_i$ being pairwise distinct. We define the solvable generalization of the Baumslag-Solitar group by
\[
\Gamma_n=\left<a,t_1,...,t_r\ |\ t_it_j=t_jt_i,\ i\neq j,\ t_ia{t_i}^{-1}=a^{{p_i}^{y_i}},\ i=1,...,r \right>.
\]

More generally, let $S=\{n_1,...,n_r\}$ be a set of pairwise coprime positive integers such that $n_i\geq 2$ for some $i$. Define
\[
\Gamma(S)=\left<a,t_1,...,t_r\ |\ t_it_j=t_jt_i,\ i\neq j,\ t_ia{t_i}^{-1}=a^{n_i},\ i=1,...,r \right>.
\]
\end{defi}
The group $\Gamma(S)$ is always torsion-free.

Note that $BS(1,n)$ is a metabelian group and it admits the following splitting
$$
1 \to \mathbb Z\left[\frac{1}{n}\right] \to BS(1,n) \stackrel{\dashleftarrow}{\to} \mathbb Z \to 1.
$$
where $\mathbb Z\left[\frac{1}{n}\right]$ denotes the $n$-adic rationals and contains the commutator subgroup $[BS(1,n), BS(1,n)]$. Similarly,
$\Gamma_n$ is characterized by the following short exact sequence
\begin{equation}\label{splitexact}
1 \to \mathbb{Z}\left[\frac{1}{n}\right] \to \Gamma_n \stackrel{\varphi}{\to} \mathbb{Z}^r \to 1.
\end{equation}
Here, $\varphi$ is the canonical projection with $a \mapsto 1$, $\mathbb{Z} \left[ \frac{1}{n} \right]= \langle a_j,\  j \in \mathbb{Z}\ |\ a_j^n=a_{j+1},\  j \in \mathbb{Z} \rangle $ and is generated by the elements
\[
a_j=(t_1...t_r)^j a(t_1...t_r)^{-j} \in \Gamma_n.
\] Using the presentation $\mathbb{Z}^r=\langle t_1,...,t_r \ |\ t_it_j=t_jt_i, i \neq j \rangle$, the exact sequence \eqref{splitexact} splits using the section $\mathbb{Z}^r \to \Gamma_n$ sending $t_i \mapsto t_i$. Thus, $\Gamma_n = \mathbb{Z}\left[\frac{1}{n}\right] \rtimes \mathbb{Z}^r$ is the semidirect product of these two subgroups, and every element $w \in \Gamma_n$ can be uniquely written as $w=t_1^{\alpha_1}...t_r^{\alpha_r}u$ for $u \in \mathbb{Z}\left[\frac{1}{n}\right]$ and $\alpha_i \in \mathbb{Z}$ (we put $u$ on the right side following the notation from Bogopolski in \cite{Bogopolski}). Observe that the ``$t_i$-coordinates'' in $\Gamma_n$ are well behaved, that is, $(t_1^{\alpha_1}...t_r^{\alpha_r}u)(t_1^{\beta_1}...t_r^{\beta_r}u')=t_1^{\alpha_1+\beta_1}...t_r^{\alpha_r+\beta_r}u''$ for some $u'' \in \mathbb{Z}\left[\frac{1}{n}\right]$. Secondly, because of the presentation of the subgroup $\mathbb{Z}\left[\frac{1}{n}\right]$, we see that any two generators $a_i,a_j$ must be powers of the common generator $a_{\min\{i,j\}}$. Note that $\mathbb{Z}\left[\frac{1}{n}\right]$ is an infinitely generated abelian group and $\Gamma_n$ is metabelian.

\begin{theorem}\label{meusigmagamman}
The complement $\Sigma^1(\Gamma(S))^c$ of the $\Sigma^1$ of the group
\[
\Gamma(S)=\left<a,t_1,...,t_r\ |\ t_it_j=t_jt_i,\ i\neq j,\ t_ia{t_i}^{-1}=a^{n_i},\ i=1,...,r \right>
\] is given by
\[
\Sigma^1(\Gamma (S))^c=\{[\chi_i]\ |\chi_i(t_i)=1 \text{and~} \chi_i(t_j)=0 \text{~for~} j\neq i\},
\] In particular, if $n=p_1^{y_1}...p_r^{y_r}$ is a prime decomposition, then 
\[
\Sigma^1(\Gamma_n)^c=\{[\chi_1],...,[\chi_r]\}.
\]
Furthermore, $\Sigma^1(\Gamma(S))^c$ lies inside an open hemisphere in $S(\Gamma(S))$.
\end{theorem}
\begin{proof}
As pointed out in \cite{BNS}, the $\Sigma^1$ coincides with $\Sigma_{G'}$ of \cite{BS}. For the metabelian group $\Gamma(S)$, the quotient $\mathbb Z^r$ is the torsion-free part of the abelianization so that $S(\Gamma (S))=S(\mathbb Z^r)$. It follows from Proposition 2.1 and formula (2.3) of \cite{BS} that 
$$
\Sigma^1(\Gamma(S))=\bigcup_{\lambda \in C(A)} \{[\chi] \in S(\Gamma(S)) \mid \chi(\lambda)>0\}
$$
where $A=Ker \varphi=\mathbb Z\left[\frac{1}{n}\right]$ as a $\mathbb Z[\mathbb Z^r]$-module and $C(A)=\{\lambda \in \mathbb Z[\mathbb Z^r] \mid \lambda \cdot \alpha =\alpha, \text{for all~} \alpha \in A\}$ is the centralizer of $A$. Let $[\chi]\in S(\Gamma (S))$. 

Case (1): If $\chi(t_i)<0$ for some $i, 1\le i\le r$ then we let $\lambda=n_it_i^{-1}$. Note that $n_it_i^{-1}\cdot a=t_i^{-1}a^{n_i}t_i=a$. It follows that $\lambda \in C(A)$ and $[\chi]\in \Sigma^1(\Gamma (S))$. 

Let $I_k=\{i_j \mid 1\le i_1 < ... < i_k \le r\}$, $k\ge 2$, be a subset of the set $I=\{1, 2, ..., r\}$. 

Case (2): If $\chi(t_{i_j})>0$ for $i_j\in I_k$ and $\chi(s)=0$ for $s\in I\setminus I_k$ then we let $\lambda=\sum_{j=1}^k \alpha_j t_{i_j}$ where $\alpha_j$ are integers such that $\alpha_1n_{1} + ... + \alpha_rn_{r}=1$ since $n_{i_1},..., n_{i_r}$ are pairwise relatively prime. It is easy to see that $\lambda \in C(A)$. 

If $\chi(\lambda)>0$ then $[\chi]\in \Sigma^1(\Gamma (S))$. 

Now suppose $\kappa=\chi(\lambda)\le 0$. Without loss of generality, we may assume that $\alpha_{i_1}>0$. Since $n_{i_1}t_{i_1}^{-1} \in C(A)$, it follows that for any integer $M$, $M\lambda - (M-1)(n_{i_1}t_{i_1}^{-1})\cdot a=a^{M-(M-1)}=a$ so that $\hat \lambda_M=M\lambda - (M-1)(n_{i_1}t_{i_1}^{-1}) \in C(A)$. Now, it is straightforward to see that $\chi(\hat \lambda_M)=(M-1)n_{i_1}\chi(t_{i_1}) + M\kappa$. There exists a positive integer $M$ such that $\chi(\hat \lambda_M)>0$. In other words, $[\chi]\in \Sigma^1(\Gamma (S))$.

Now, the set of characters that do not belong to Case (1) or Case (2) is $\{[\chi_i]\}$, where $\chi_i(t_i)=1$ and $\chi_i(t_j)=0$ if $j\neq i$. To see that this set is the complement of $\Sigma^1(\Gamma (S))$, it suffices to show that $[\chi_i]\in \Sigma^1(\Gamma (S))^c$ for each $i$. Observe that if $\gamma=\sum c_jt_j^{q_j} \in C(A)$ then either all $q_j >0$ when $c_j\ne 0$ or for some $j$, $c_j=n_j$ and $q_j=-1$ with $q_i=0$ for $i\ne j$. Thus,  $\chi_i(\gamma)=c_iq_i$ cannot be positive so each $[\chi_i]\notin \Sigma^1(\Gamma (S))$.

\end{proof}

\begin{remark}  In an earlier version of this paper, Theorem \ref{meusigmagamman} was first proved using a general geometric argument \cite[Theorem A3.1]{Strebel}.
\end{remark}

For the remaining of this paper, we focus on the groups $\Gamma_n$.

\section{Finite index subgroups of $\Gamma_n$}

In this section we study the finite index subgroups $H$ of $\Gamma_n$. First, in Theorem \ref{fisgamman} we find a specific set of generators for $H$ using a generalization of an argument given by Bogopolski in \cite{Bogopolski}. We use these generators to compute the index of $H$ in $\Gamma_n$. Then, in Theorem \ref{presforh}, we give a presentation for $H$ and, in Theorem \ref{meusigmah}, we compute $\Sigma^1(H)$. We end the section by exhibiting finite index subgroups $H$ of $\Gamma_n$ which are not isomorphic to $\Gamma_k$ for any $k \geq 2$.

\subsection{Generators, cosets and index}

The following useful lemma has an elementary proof and was used by Bogopolski in \cite{Bogopolski}.

\begin{lema}\label{itnlemma}
Let $n,s \geq 1$ be integers. Let $m$ be the biggest positive divisor of $s$ such that $\gcd(m,n)=1$. Then $s$ divides $mn^s$.
\end{lema}

To facilitate our computation, we aim to find a {\it good} set of generators of a finite index subgroup of $\Gamma_n$. To do so, we need the next two lemmas.

\begin{lema}[Replacing $j_0$ by any $j$]\label{replj0byjgen}
Suppose
\begin{equation}\label{replacej0eqn}
H= \langle {t_1}^{k_{11}}...{t_r}^{k_{1r}}a_{q_1}^{l_1},{t_2}^{k_{22}}...{t_r}^{k_{2r}}a_{q_2}^{l_2},...,{t_r}^{k_{rr}}a_{q_r}^{l_r},a_{j_0}^l \rangle \leq \Gamma_n
\end{equation}is a subgroup with arbitrary integers $k_{ii},l >0$, $k_{ij}\geq 0$ and $q_i,l_i,j_0 \in \Z$. Then, for any chosen $j \in \Z$, we can replace $a_{j_0}^l$ above by $a_j^l$, up to modifying $l>0$ by another positive integer (also called $l$), that is, $H= \langle {t_1}^{k_{11}}...{t_r}^{k_{1r}}a_{q_1}^{l_1},{t_2}^{k_{22}}...{t_r}^{k_{2r}}a_{q_2}^{l_2},...,{t_r}^{k_{rr}}a_{q_r}^{l_r},a_j^l \rangle$.
\end{lema}

\begin{proof}
If $j \leq j_0$ we know from the presentation of $\Z\left[\frac{1}{n}\right]$ that $a_{j_0}$ is a positive power of $a_j$, so $a_{j_0}^l$ is also a positive power of $a_j$ and the lemma is obviously true. Let us treat the case $j>j_0$. Using that $\Z\left[\frac{1}{n}\right]$ is abelian and the relations of $\Gamma_n$, we can show that
\[
({t_i}^{k_{ii}}...{t_r}^{k_{ir}}a_{q_i}^{l_i})^{m_i}a_{j_0}^l({t_i}^{k_{ii}}...{t_r}^{k_{ir}}a_{q_i}^{l_i})^{-m_i}=a_{j_0}^{l{p_i}^{m_iy_ik_{ii}}...{p_r}^{m_iy_rk_{ir}}}
\]for every $i$ and every integer $m_i>0$. Thus we can replace $a_{j_0}^l$ in the expression of $H$ by this element $a_{j_0}^{l{p_i}^{m_iy_ik_{ii}}...{p_r}^{m_iy_rk_{ir}}}$, that is, we can multiply the power $l$ of $a_{j_0}$ by ${p_i}^{m_iy_ik_{ii}}...{p_r}^{m_iy_rk_{ir}}$ in \eqref{replacej0eqn}, and since this new power is still positive we can repeat the process recursively. By doing this for $i=1,...,r$ we can replace the power $l$ of $a_{j_0}$ in \eqref{replacej0eqn} by any number of the form
\[
l({p_1}^{m_1y_1k_{11}}...{p_r}^{m_1y_rk_{1r}})({p_2}^{m_2y_2k_{22}}...{p_r}^{m_2y_2k_{2r}})...({p_r}^{m_ry_rk_{rr}})
\]for any $m_1,...,m_r >0$. By putting together the first primes in the parentheses we rewrite this as
\[
{p_1}^{m_1y_1k_{11}}{p_2}^{m_2y_2k_{22}}...{p_r}^{m_ry_rk_{rr}}l\lambda
\]for some integer $\lambda>0$ depending on the $m_i$. In particular, for the integers $m_i=k_{11}...\widehat{k_{ii}}...k_{rr}$ we can replace the power $l$ of $a_{j_0}$ by
\[
{p_1}^{y_1k}{p_2}^{y_2k}...{p_r}^{y_rk} l\lambda =n^k l \lambda,
\]where $k=k_{11}...k_{rr}$. But $a_{j_0}^{n^kl\lambda}=a_{j_0+k}^{l\lambda}$, which is a positive power of $a_{j_0+1}$. We repeat this process a finite number of times until we reach the index $j>j_0$ we wanted and the lemma is proved.
\end{proof}

\begin{lema}[Replacing $l$ by $m$]\label{repllbymgen}
Let
\begin{equation}\label{repllbymeqngen}
H= \langle {t_1}^{k_{11}}...{t_r}^{k_{1r}}a_{q_1}^{l_1},{t_2}^{k_{22}}...{t_r}^{k_{2r}}a_{q_2}^{l_2},...,{t_r}^{k_{rr}}a_{q_r}^{l_r},a_j^l \rangle \leq \Gamma_n
\end{equation}be a subgroup with arbitrary integers $k_{ii},l >0$, $k_{ij}\geq 0$ and $q_i,l_i,j \in \Z$. Let $m$ be the biggest divisor of $l$ such that $\gcd(m,n)=1$. Then we can replace $a_j^l$ by $a_j^m$ in the expression above, that is, $H= \langle {t_1}^{k_{11}}...{t_r}^{k_{1r}}a_{q_1}^{l_1},{t_2}^{k_{22}}...{t_r}^{k_{2r}}a_{q_2}^{l_2},...,{t_r}^{k_{rr}}a_{q_r}^{l_r},a_j^m \rangle$.
\end{lema}

\begin{proof}
It suffices to show that the inclusions $a_j^l \in \langle {t_1}^{k_{11}}...{t_r}^{k_{1r}}a_{q_1}^{l_1},{t_2}^{k_{22}}...{t_r}^{k_{2r}}a_{q_2}^{l_2},...,{t_r}^{k_{rr}}a_{q_r}^{l_r},a_j^m \rangle$ and $a_j^m \in H$ hold. The first inclusion is straightforward, because $l$ is a multiple of $m$ and so $a_j^l$ is a power of $a_j^m$. For the second inclusion first observe that by Lemma \ref{itnlemma}, $l$ must divide $mn^l$ and so it must also divide $mn^{lk_{rr}}$ . This implies that the number
\[
\gamma=\frac{mn^{lk_{rr}}{p_1}^{y_1(k_{11}-1)lk_{rr}}...{p_{r-1}}^{y_{r-1}(k_{r-1,r-1}-1)lk_{rr}}\prod_{j=1}^{r-1}\prod_{i=j+1}^{r}{p_i}^{y_ik_{ji}k_{rr}l}}{l}
\] is an integer. Let $A_1,...,A_r$ be the first $r$ generators of $H$ in (\ref{repllbymeqngen}), that is, $H=\langle A_1,...,A_r,a_j^l \rangle$. It is straightforward to show that
\[
{A_1}^{-lk_{rr}}...{A_{r-1}}^{-lk_{rr}}{A_r}^{-l}(a_j^l)^\gamma{A_r}^l{A_{r-1}}^{lk_{rr}}...{A_1}^{lk_{rr}}=a_j^m,
\]then $a_j^m \in H$, as desired.
\end{proof}

\begin{theorem}\label{fisgamman}
For any $\Gamma_n$, the following properties hold.
\begin{itemize}
\item[$1)$] Every finite index subgroup $H$ of $\Gamma_n$ can be written as
\[
H=\langle {t_1}^{k_{11}}...{t_r}^{k_{1r}}a^{l_1},{t_2}^{k_{22}}...{t_r}^{k_{2r}}a^{l_2},...,{t_r}^{k_{rr}}a^{l_r},a^m\rangle\ \ \ \ \ \ (*)
\]for $0 \leq k_{1i},...,k_{i-1,i}<k_{ii}$, $l_i \in \Z$ and $m>0$ an integer such that $\gcd(m,n)=1$ and $H \cap \left< a \right>= \left< a^m \right>$.

\item[$2)$]If $H$ is any subgroup of $\Gamma_n$ given by the expression $(*)$ for $0 \leq k_{1i},...,k_{i-1,i}<k_{ii}$, $l_i \in \Z$ and $m>0$ such that $\gcd(m,n)=1$ and $H \cap \left< a \right>= \left< a^m \right>$, then $T=\{{t_1}^{\beta_1}...{t_r}^{\beta_r}a^j\ |\ 0 \leq \beta_i < k_{ii},\ 0 \leq j <m\}$ is a transversal of $H$ in $\Gamma_n$. In particular, the index of $H$ in $\Gamma_n$ is $k_{11}...k_{rr}m$ and $H$ has finite index in $\Gamma_n$.
\end{itemize}
\end{theorem}

\begin{proof}
$1)$ First, since $\Gamma_n$ is finitely generated and $H$ is finite index, by the Reidemeister-Schreier theorem $H$ must be also finitely generated and we write
\[
H=\langle {t_1}^{\alpha_{11}}...{t_r}^{\alpha_{1r}}v_1,..., {t_1}^{\alpha_{m1}}...{t_r}^{\alpha_{mr}}v_m \rangle
\]for $\alpha_{ij} \in \Z$ and $v_i \in \Z\left[\frac{1}{n}\right]$. Note that $m \geq r$. Otherwise, $\varphi(H)$ would be a subgroup of $\Z^r$ with rank $<r$ and then would have infinite index, a contradiction because $\varphi$ is surjective. With a similar projection argument, we see that there must be at least one $i$ such that $\alpha_{i1} \neq 0$. Let $\displaystyle{k_{11}=\gcd_{\alpha_{i1} \neq 0}\{\alpha_{i1}\}}$. Since $k_{11}>0$ is the smallest positive integer combination of the $\alpha_{i1} \neq 0$, we can obtain inside $H$ an element of the form ${t_1}^{k_{11}}...{t_r}^{k_{1r}}u_1$ for some $k_{12},...,k_{1r} \in \Z$ and $u_1 \in \Z\left[\frac{1}{n}\right]$, so we can write
\begin{equation}\label{add1}
H=\langle {t_1}^{\alpha_{11}}...{t_r}^{\alpha_{1r}}v_1,..., {t_1}^{\alpha_{m1}}...{t_r}^{\alpha_{mr}}v_m, {t_1}^{k_{11}}...{t_r}^{k_{1r}}u_1 \rangle.
\end{equation} Now, since all the nonzero $\alpha_{i1}$ are multiples of $k_{11}$, say, $\alpha_{i1}=d_ik_{11}$, we can replace ${t_1}^{\alpha_{i1}}...{t_r}^{\alpha_{ir}}v_i$ by $({t_1}^{\alpha_{i1}}...{t_r}^{\alpha_{ir}}v_i)({t_1}^{k_{11}}...{t_r}^{k_{1r}}u_1)^{-d_i}={t_2}^{\alpha_{i2}'}...{t_r}^{\alpha_{ir}'}v_i'$ in (\ref{add1}). Then, after relabeling these new generators, we can write
\[
H=\langle {t_2}^{\alpha_{12}}...{t_r}^{\alpha_{1r}}v_1,..., {t_2}^{\alpha_{m2}}...{t_r}^{\alpha_{mr}}v_m, {t_1}^{k_{11}}...{t_r}^{k_{1r}}u_1 \rangle.
\] We added a new generator and ``eliminated'' all the $t_1$ coordinates of the first $m$ generators of $H$. This was the first step. In a similar way, we can do this for all the other $t_2,...,t_r$ coordinates. After $r$ steps, we added $r$ new generators and eliminated all the $t_1,...,t_r$ letters from the first $m$ generators from $H$, so we have
\[
H= \langle v_1,...,v_m, {t_1}^{k_{11}}...{t_r}^{k_{1r}}u_1,{t_2}^{k_{22}}...{t_r}^{k_{2r}}u_2,...,{t_r}^{k_{rr}}u_r\rangle
\]with $k_{ii}>0$ and $v_i,u_i \in \Z\left[\frac{1}{n}\right]$. But in $\Z\left[\frac{1}{n}\right]$ we have $\langle v_1,...,v_m \rangle=\langle u \rangle$ for some $u \in \Z\left[\frac{1}{n}\right]$ and
\begin{equation}\label{addr}
H= \langle {t_1}^{k_{11}}...{t_r}^{k_{1r}}u_1,{t_2}^{k_{22}}...{t_r}^{k_{2r}}u_2,...,{t_r}^{k_{rr}}u_r,u \rangle
\end{equation}

By manipulating the generators above if necessary, we may suppose that $0 \leq k_{1i},...,k_{i-1,i}<k_{ii}$ (they could be also positive if we wanted) in (\ref{addr}). Finally, write $u_i=a_{q_i}^{l_i}, u=a_q^l$ for $q_i,q,l_i,l \in \Z$. Then
\begin{equation}\label{halfproof}
H= \langle {t_1}^{k_{11}}...{t_r}^{k_{1r}}a_{q_1}^{l_1},{t_2}^{k_{22}}...{t_r}^{k_{2r}}a_{q_2}^{l_2},...,{t_r}^{k_{rr}}a_{q_r}^{l_r},a_q^l \rangle.
\end{equation}

Let us show that we may assume $l>0$ above. If $l \neq 0$ then, up to changing $a_q^l$ by $(a_q^l)^{-1}=a_q^{-1}$ if necessary, we are done. If $l=0$, that is,
\begin{equation}\label{nol}
H= \langle {t_1}^{k_{11}}...{t_r}^{k_{1r}}a_{q_1}^{l_1},{t_2}^{k_{22}}...{t_r}^{k_{2r}}a_{q_2}^{l_2},...,{t_r}^{k_{rr}}a_{q_r}^{l_r} \rangle,
\end{equation} we do the following: since $\Z^r$ is abelian, every commutator of elements in $H$ must be in $Ker(\varphi)$ (and obviously in $H$). At least one of the commutators between the $r$ generators of $H$ in (\ref{nol}) must be non-trivial. Otherwise, $H$ would be a finite index abelian subgroup of $\Gamma_n$ and we would have $\Sigma^1(\Gamma_n)=S(\Gamma_n)$ by using Proposition $B1.11$ in \cite{Strebel}, a contradiction to Theorem \ref{meusigmagamman}. Then let $a_j^{l'}$ ($l' \neq 0$) be a non-trivial commutator between two generators of $H$. We can add it to \ref{nol} and up to changing $a_j^{l'}$ by its inverse, we are done.

Our next steps will be eliminating the subindices $q_i$ from the $a$ letters in the generators of (\ref{halfproof}). Fix some $1 \leq i \leq r$. If $q_i \geq 0$, then $a_{q_i}^{l_i}$ is a power of $a$ and we are done by doing this replacement in (\ref{halfproof}). Suppose $q_i<0$. By Lemma \ref{replj0byjgen} we replace $q$ by $q_i$ in (\ref{halfproof}). Now, let $m$ be the biggest divisor of $l$ such that $\gcd(m,n)=1$. By Lemma \ref{repllbymgen} we can also replace $l$ by $m$ above and obtain
\[
H= \langle {t_1}^{k_{11}}...{t_r}^{k_{1r}}a_{q_1}^{l_1},{t_2}^{k_{22}}...{t_r}^{k_{2r}}a_{q_2}^{l_2},...,{t_r}^{k_{rr}}a_{q_r}^{l_r},a_{q_i}^m \rangle.
\] Since $\gcd(m,n)=1$ we also have $\gcd(m,n^{-q_i})=1$ and there must be $\tilde{\alpha},\tilde{\beta} \in \Z$ such that $\tilde{\alpha}m+\tilde{\beta}n^{-q_i}=1$. Then for $\alpha=l_i\tilde{\alpha}$ and $\beta=l_i\tilde{\beta}$ we have $\alpha m+\beta n^{-q_i}=l_i$, or
\[
l_i-m\alpha= n^{-q_i}\beta.
\] Then, using the relations in $\Gamma_n$ we have
\begin{eqnarray*}
H &=& \langle {t_1}^{k_{11}}...{t_r}^{k_{1r}}a_{q_1}^{l_1},{t_2}^{k_{22}}...{t_r}^{k_{2r}}a_{q_2}^{l_2},...,{t_i}^{k_{ii}}...{t_r}^{k_{ir}}a_{q_i}^{l_i},...,{t_r}^{k_{rr}}a_{q_r}^{l_r},a_{q_i}^m \rangle \\
&=& \langle {t_1}^{k_{11}}...{t_r}^{k_{1r}}a_{q_1}^{l_1},{t_2}^{k_{22}}...{t_r}^{k_{2r}}a_{q_2}^{l_2},...,{t_i}^{k_{ii}}...{t_r}^{k_{ir}}a_{q_i}^{l_i-m\alpha},...,{t_r}^{k_{rr}}a_{q_r}^{l_r},a_{q_i}^m \rangle \\
&=& \langle {t_1}^{k_{11}}...{t_r}^{k_{1r}}a_{q_1}^{l_1},{t_2}^{k_{22}}...{t_r}^{k_{2r}}a_{q_2}^{l_2},...,{t_i}^{k_{ii}}...{t_r}^{k_{ir}}a_{q_i}^{n^{-q_i}\beta},...,{t_r}^{k_{rr}}a_{q_r}^{l_r},a_{q_i}^m \rangle \\
&=& \langle {t_1}^{k_{11}}...{t_r}^{k_{1r}}a_{q_1}^{l_1},{t_2}^{k_{22}}...{t_r}^{k_{2r}}a_{q_2}^{l_2},...,{t_i}^{k_{ii}}...{t_r}^{k_{ir}}a^\beta,...,{t_r}^{k_{rr}}a_{q_r}^{l_r},a_{q_i}^m \rangle \\
\end{eqnarray*} and relabeling $\beta$ by $l_i$, $m$ by $l$ and $q_i$ by $q$ again we have
\[
H=\langle {t_1}^{k_{11}}...{t_r}^{k_{1r}}a_{q_1}^{l_1},{t_2}^{k_{22}}...{t_r}^{k_{2r}}a_{q_2}^{l_2},...,{t_i}^{k_{ii}}...{t_r}^{k_{ir}}a^{l_i},...,{t_r}^{k_{rr}}a_{q_r}^{l_r},a_q^l \rangle,
\]that is, we removed the subindex $q_i$ from $a_{q_i}^{l_i}$ in \ref{halfproof}. If we do this for all $i$ we remove all the subindices and obtain
\[
H=\langle {t_1}^{k_{11}}...{t_r}^{k_{1r}}a^{l_1},{t_2}^{k_{22}}...{t_r}^{k_{2r}}a^{l_2},...,{t_r}^{k_{rr}}a^{l_r},a_q^l\rangle
\] for some $q \in \Z$. We can use Lemma \ref{replj0byjgen} to replace $q$ by $0$ and we get the desired set of generators for $H$. To finish, let $m$ (a new one) be the biggest divisor of $l$ such that $\gcd(m,n)=1$. By Lemma \ref{repllbymgen}, we replace $a^l$ by $a^m$ in the expression above. If $H \cap \left< a \right>= \left< a^m \right>$, we are done. If not, let $m'=\min\{k \geq 1\ |\ a^k \in H\}$. It's easy to see that $H \cap \left< a \right>= \left< a^{m'} \right>$. Since $a^m \in H$, $m$ is a multiple of $m'$ and we have $\gcd(m',n)=1$. Then, by adding $a^{m'}$ to the set of generators of $H$, the generator $a^m$ can be removed. By relabeling $m'$ by $m$, we obtain the desired result.

\bigskip

$2)$ Let $H$ be such a subgroup. As shown in item $1)$, we may suppose that $k_{ij} > 0$ for all $i,j$. Let us first show that $\Gamma_n=\bigcup_{{t_1}^{\beta_1}...{t_r}^{\beta_r}a^j \in T} H{t_1}^{\beta_1}...{t_r}^{\beta_r}a^j$.  Every element of $\Gamma_n$ is written as ${t_1}^{-\alpha_1}...{t_r}^{-\alpha_r}a^l{t_1}^{\gamma_1}...{t_r}^{\gamma_r}$ for $\alpha_i,\gamma_i \geq 0$ and $l \in \Z$. Since $k_{ij} > 0$ for all $i,j$, one can show that every coset of $\Gamma_n$ is of the form $Ha^l{t_1}^{\gamma_1}...{t_r}^{\gamma_r}$ for $l \in \Z$ and $\gamma_i \geq 0$. Now we claim that every such coset can be also written as $H{t_1}^{\gamma_1}...{t_r}^{\gamma_r}a^{l'}$ for some integer $l'$. In fact, because $1=\gcd(m,n)=\gcd(m,{p_1}^{y_1}...{p_r}^{y_r})$, the prime decomposition of $m$ does not involve any of the $p_i$. Then it is also true that $\gcd(m,{p_1}^{\gamma_1y_1}...{p_r}^{\gamma_ry_r})=1$. Let $k,k'$ be integers such that $km+k'{p_1}^{\gamma_1y_1}...{p_r}^{\gamma_ry_r}=1$. Then $l+(-lk)m=(lk'){p_1}^{\gamma_1y_1}...{p_r}^{\gamma_ry_r}$ and relabeling $-lk$ by $k$ and $lk'$ by $k'$ we get $l+km=k'{p_1}^{\gamma_1y_1}...{p_r}^{\gamma_ry_r}$. Now since $a^m \in H$ we do

\begin{eqnarray*}
Ha^l{t_1}^{\gamma_1}...{t_r}^{\gamma_r}	&=& H(a^m)^ka^l{t_1}^{\gamma_1}...{t_r}^{\gamma_r}\\
										&=& Ha^{l+km}{t_1}^{\gamma_1}...{t_r}^{\gamma_r}\\
										&=& Ha^{k'{p_1}^{\gamma_1y_1}...{p_r}^{\gamma_ry_r}}{t_1}^{\gamma_1}...{t_r}^{\gamma_r}\\
										&=& H{t_1}^{\gamma_1}...{t_r}^{\gamma_r}a^{k'}
\end{eqnarray*}and relabeling $k'$ by $l'$ we showed the claim. To transform this coset into one of the cosets in the theorem, we apply successive algorithms: choose some index $i$. If $\gamma_i < k_{ii}$ we stop the algorithm. If $\gamma_i \geq k_{ii}$, by manipulating this coset we show that
\[
H{t_1}^{\gamma_1}...{t_r}^{\gamma_r}a^l=H{t_1}^{\gamma_1}...{t_{i-1}}^{\gamma_{i-1}}{t_i}^{\gamma_i-k_{ii}}{t_{i+1}}^{\gamma_{i+1}'}...{t_r}^{\gamma_r'}a^{l'}
\]for some integer $l'$. If $\gamma_i-k_{ii} < k_{ii}$ we stop the algorithm. If $\gamma_i-k_{ii} \geq k_{ii}$ we do the above again. Then after finite steps our ``$i$-algorithm'' shows that
\[
H{t_1}^{\gamma_1}...{t_r}^{\gamma_r}a^l=H{t_1}^{\gamma_1}...{t_{i-1}}^{\gamma_{i-1}}{t_i}^{\beta_i}{t_{i+1}}^{\gamma_{i+1}'}...{t_r}^{\gamma_r'}a^{l'}
\] for some $0 \leq \beta_i <k_{ii}$. Now, starting with the coset $H{t_1}^{\gamma_1}...{t_r}^{\gamma_r}a^l$, we successively apply the ``$i$-algorithm'' for $i=1,2,...,r$ and obtain exactly
\[
H{t_1}^{\gamma_1}...{t_r}^{\gamma_r}a^l=H{t_1}^{\beta_1}...{t_r}^{\beta_r}a^{l'}
\] for $0 \leq \beta_i <k_{ii}$ and $l' \in \Z$. Finally, write $l'=qm+j$ for $0 \leq j <m$. Then $H{t_1}^{\beta_1}...{t_r}^{\beta_r}a^{l'}=H{t_1}^{\beta_1}...{t_r}^{\beta_r}a^j$ because
\begin{eqnarray*}
{t_1}^{\beta_1}...{t_r}^{\beta_r}a^{l'}({t_1}^{\beta_1}...{t_r}^{\beta_r}a^j)^{-1}&=&{t_1}^{\beta_1}...{t_r}^{\beta_r}a^{l'-j}{t_r}^{-\beta_r}...{t_1}^{-\beta_1}\\
&=& {t_1}^{\beta_1}...{t_r}^{\beta_r}a^{mq}{t_r}^{-\beta_r}...{t_1}^{-\beta_1}\\
&=& (a^m)^{qp_1^{\beta_1y_1}...{p_r}^{\beta_ry_r}} \in H.\\
\end{eqnarray*} This shows that $\Gamma_n=\bigcup_{{t_1}^{\beta_1}...{t_r}^{\beta_r}a^j \in T} H{t_1}^{\beta_1}...{t_r}^{\beta_r}a^j$.

Now let us show that the cosets over $T$ are all distinct. Let $H{t_1}^{\beta_1}...{t_r}^{\beta_r}a^j=H{t_1}^{\beta_1'}...{t_r}^{\beta_r'}a^{j'}$ for $0 \leq \beta_i, \beta_i' <k_{ii}$ and $0 \leq j,j' < m$. By definition,
\begin{eqnarray*}
w=a^{{p_1}^{y_1\beta_1}...{p_r}^{y_r\beta_r}(j-j')}{t_1}^{\beta_1-\beta_1'}...{t_r}^{\beta_r-\beta_r'}&=&{t_1}^{\beta_1}...{t_r}^{\beta_r}a^{j-j'}{t_1}^{-\beta_1}...{t_r}^{-\beta_r}{t_1}^{\beta_1-\beta_1'}...{t_r}^{\beta_r-\beta_r'}\\
&=&{t_1}^{\beta_1}...{t_r}^{\beta_r}a^j({t_1}^{\beta_1'}...{t_r}^{\beta_r'}a^{j'})^{-1} \in H.
\end{eqnarray*} Then, projecting in $\Z^r$,
\[
(\beta_1-\beta_1',...,\beta_r-\beta_r')=\varphi(w) \in \varphi(H)= \left< (k_{11},k_{12},...,k_{1r}),(0,k_{22},...,k_{2r}),...,(0,...,0,k_{rr}) \right>.
\]Write
\[
(\beta_1-\beta_1',...,\beta_r-\beta_r')= \lambda_1(k_{11},k_{12},...,k_{1r})+\lambda_2(0,k_{22},...,k_{2r})+...+\lambda_r(0,...,0,k_{rr})
\]for integers $\lambda_i$. Since the first vector $(k_{11},k_{12},...,k_{1r})$ is the only one with non-vanishing first coordinate we have $\beta_1-\beta_1'=\lambda_1k_{11}$. Since $0 \leq \beta_1, \beta_1' <k_{11}$ we must have $\beta_1=\beta_1'$ and therefore $\lambda_1=0$. By easy induction we can show that all the $\lambda_i$ must vanish. Now, we just have to show that $j=j'$. We already have $a^{{p_1}^{y_1\beta_1}...{p_r}^{y_r\beta_r}(j-j')} \in H$. Since $H \cap \left< a \right>=\left< a^m \right>$ (by item $1)$), we have
\[
{p_1}^{y_1\beta_1}...{p_r}^{y_r\beta_r}(j-j')=qm
\]for some $q \in \Z$. So $m$ divides ${p_1}^{y_1\beta_1}...{p_r}^{y_r\beta_r}(j-j')$. Since $\gcd(n,m)=1$, $m$ does not contain any of the $p_i$ in its prime decomposition, and therefore $m$ must divide $j-j'$. Since $0 \leq j,j'<m$ we have $j=j'$, as desired. This completes the proof.
\end{proof}

\subsection{A presentation}

We now give a presentation for an arbitrary finite index subgroup $H$ of $\Gamma_n$.

\begin{theorem}\label{presforh} Let $H$ be any finite index subgroup of $\Gamma_n$ (see Theorem \ref{fisgamman}), say,
\[
H=\langle {t_1}^{k_{11}}...{t_r}^{k_{1r}}a^{l_1},{t_2}^{k_{22}}...{t_r}^{k_{2r}}a^{l_2},...,{t_r}^{k_{rr}}a^{l_r},a^m\rangle\ \ \ \ \ \ (*)
\]for $k_{ii}>0$, $k_{ij}\geq 0$, $l_i \in \Z$ and $m>0$ an integer such that $\gcd(m,n)=1$ and $H \cap \left< a \right>= \left< a^m \right>$. Then $H$ has the following presentation:
\[
H \simeq \left<\alpha,x_1,...,x_r\ |\ x_i\alpha x_i^{-1}=\alpha^{P_i},\ x_ix_jx_i^{-1}x_j^{-1}=\alpha^{R_{ij}} \right>,
\]where $P_i=p_i^{y_ik_{ii}}...p_r^{y_rk_{ir}}$ ($i=1,...,r$) and $R_{ij} \in \Z$ characterized by
\[
l_iP_i(1-P_j)-l_jP_j(1-P_i)=R_{ij}m.
\]
\end{theorem}

\begin{proof}
It is easy to see that $(t_i^{k_{ii}}...t_r^{k_{ir}}a^{l_i})a^m(t_i^{k_{ii}}...t_r^{k_{ir}}a^{l_i})^{-1}=a^{mP_i}$ in $\Gamma_n$, for $i=1,...,r$. Also, since
\[
(t_i^{k_{ii}}...t_r^{k_{ir}}a^{l_i})(t_j^{k_{jj}}...t_r^{k_{jr}}a^{l_j})(t_i^{k_{ii}}...t_r^{k_{ir}}a^{l_i})^{-1}(t_j^{k_{jj}}...t_r^{k_{jr}}a^{l_j})^{-1}=a^{l_iP_i(1-P_j)-l_jP_j(1-P_i)} \in H \cap \left< a \right>=\left< a^m \right>,
\]we have $l_iP_i(1-P_j)-l_jP_j(1-P_i)=R_{ij}m$ for some integer $R_{ij}$.

We write 
$(t_i^{k_{ii}}...t_r^{k_{ir}}a^{l_i})(t_j^{k_{jj}}...t_r^{k_{jr}}a^{l_j})(t_i^{k_{ii}}...t_r^{k_{ir}}a^{l_i})^{-1}(t_j^{k_{jj}}...t_r^{k_{jr}}a^{l_j})^{-1}=a^{mR_{ij}}$. Now define a group
\[
G=\left<\alpha,x_1,...,x_r\ |\ x_i\alpha x_i^{-1}=\alpha^{P_i},\ x_ix_jx_i^{-1}x_j^{-1}=\alpha^{R_{ij}} \right>.
\] The group $G$ has the relations
\[
x_i\alpha=\alpha^{P_i}x_i,\ x_i\alpha^{-1}=\alpha^{-P_i}x_i,\ x_ix_j=\alpha^{R_{ij}}x_jx_i,\ x_ix_j^{-1}=x_j^{-1}\alpha^{-R_{ij}}x_i,
\]which shows that, for every fixed $i$, all the $x_i$-letters in a word with positive power can be pushed right as much as we want. Similarly, the relations
\[
\alpha x_i^{-1}=x_i^{-1}\alpha^{P_i},\ \alpha^{-1}x_i^{-1}=x_i^{-1}\alpha^{-P_i},\ x_jx_i^{-1}=x_i^{-1}\alpha^{R_{ij}}x_j,\ x_j^{-1}x_i^{-1}=x_i^{-1}x_j^{-1}\alpha^{-R_{ij}}
\] show that all the $x_i$-letters in a word with negative power can be pushed left as much as we want. Because of this, any element of $G$ is of the form $x_1^{-\lambda_1}...x_r^{-\lambda_r}\alpha^M x_r^{\delta_r}...x_1^{\delta_1}$ for $\lambda_i,\delta_i \geq 0$ and $M \in \Z$. Now let us show that $G \simeq H$. Define $\theta:G \to \Gamma_n$ by putting $\theta(\alpha)=a^m$ and $\theta(x_i)=t_i^{k_{ii}}...t_r^{k_{ir}}a^{l_i}$ for $i=1,...,r$. It is easy to check that $\theta$ is a group homomorphism and surjective, so we only need to show that $\theta$ is also injective. Indeed, let $w=x_1^{-\lambda_1}...x_r^{-\lambda_r}\alpha^M x_r^{\delta_r}...x_1^{\delta_1} \in G$ such that $\theta(w)=1$. Then
\[
{(t_1^{k_{11}}...t_r^{k_{1r}}a^{l_1})}^{-\lambda_1}...{(t_r^{k_{rr}}a^{l_r})}^{-\lambda_r}a^{mM} {(t_r^{k_{rr}}a^{l_r})}^{\delta_r}...{(t_1^{k_{11}}...t_r^{k_{1r}}a^{l_1})}^{\delta_1}=1.
\]By projecting both sides of equation above on the $t_1$-coordinate by the homomorphism $w \mapsto (w)^{t_1}$, we get $k_{11}(\delta_1-\lambda_1)=0$ and so $\delta_1=\lambda_1$. Then by conjugating the above equation on both sides by $(t_1^{k_{11}}...t_r^{k_{1r}}a^{l_1})^{\lambda_1}$ we get
\[
{(t_2^{k_{22}}...t_r^{k_{2r}}a^{l_2})}^{-\lambda_2}...{(t_r^{k_{rr}}a^{l_r})}^{-\lambda_r}a^{mM} {(t_r^{k_{rr}}a^{l_r})}^{\delta_r}...{(t_2^{k_{22}}...t_r^{k_{2r}}a^{l_2})}^{\delta_2}=1.
\] By doing this recursively we get $\delta_i=\lambda_i$ for $i=1,...,r$ and $a^{mM}=1$. Then $M=0$ (since $a$ is torsion free and $m>0$). Thus $w=x_1^{-\lambda_1}...x_r^{-\lambda_r}\alpha^0 x_r^{\lambda_r}...x_1^{\lambda_1}=1$, as desired. This completes the proof.
\end{proof}

\subsection{The $\Sigma^1$ invariant}

Let $H$ be a finite index subgroup of $\Gamma_n$, say,
\[
H=\langle {t_1}^{k_{11}}...{t_r}^{k_{1r}}a^{l_1},{t_2}^{k_{22}}...{t_r}^{k_{2r}}a^{l_2},...,{t_r}^{k_{rr}}a^{l_r},a^m\rangle\ \ \ \ \ \ (*)
\]for $k_{ii}>0$, $k_{ij}\geq 0$, $l_i \in \Z$ and $m>0$ an integer such that $\gcd(m,n)=1$ and $H \cap \left< a \right>= \left< a^m \right>$. By Theorem \ref{presforh}, we write $H$ as
\[
H=\left<\alpha,x_1,...,x_r\ |\ x_i\alpha x_i^{-1}=\alpha^{P_i},\ x_ix_jx_i^{-1}x_j^{-1}=\alpha^{R_{ij}} \right>,
\] for $P_i=p_i^{y_ik_{ii}}...p_r^{y_rk_{ir}}$ ($i=1,...,r$) and some $R_{ij} \in \Z$. Here, $\alpha=a^m$ and $x_i={t_i}^{k_{ii}}...{t_r}^{k_{ir}}a^{l_i}$. Since all the $p_i^{y_i}$ are $\geq 2$, obviously the $P_i$ also are $\geq 2$ and so it is easy to see that $\alpha$ must have torsion in the abelianized group $H^{ab}$. The $x_i$ are torsion-free, though. So we have the homeomorphism

\begin{align*}
  &\mathfrak h: S(H)\longrightarrow S^{r-1} \\
  &[\chi]\longmapsto \frac{(\chi(x_1),...,\chi(x_r))}{\Vert (\chi(x_1),...,\chi(x_r)) \Vert}.\\
\end{align*}

To compute $\Sigma^1(H)$ inside this sphere, we will use the following fact.

\begin{prop}\label{sigmafindhspecial}
Let $G$ be a finitely generated group and $H \leq G$ a finite index subgroup with inclusion $i:H \to G$ and induced map $i^*:S(G) \to S(H)$, $i^*[\chi]=[\chi \circ i]=[\chi|_H]$. Suppose that any homomorphism $\chi:H \to \R$ can be extended to a homomorphism $\hat{\chi}:G \to \R$. Then
\[
\Sigma^1(H)=i^*(\Sigma^1(G))\ \text{and}\ \Sigma^1(H)^c=i^*(\Sigma^1(G)^c).
\]
\end{prop}

\begin{proof}
By Proposition $B1.11$ in \cite{Strebel}, for any $[\chi] \in S(G)$ we have $[\chi] \in \Sigma^1(G) \Leftrightarrow [\chi|_H] \in \Sigma^1(H)$. Then $i^*(\Sigma^1(G)) \subset \Sigma^1(H)$. On the other hand, let $[\chi] \in \Sigma^1(H)$ and let $\hat{\chi}:G \to \R$ be an extension of $\chi$. We have $[\hat{\chi}|_H]=[\chi] \in \Sigma^1(H)$, so again by Proposition $B1.11$ in \cite{Strebel} we have $[\hat{\chi}] \in \Sigma^1(G)$. Then $[\chi]=i^*[\hat{\chi}] \in i^*(\Sigma^1(G))$, as desired. The other equality is similar.
\end{proof}

\begin{lema}\label{extensionlemma}
Let $H$ be a finite index subgroup of $\Gamma_n$, say,
\[
H=\langle {t_1}^{k_{11}}...{t_r}^{k_{1r}}a^{l_1},{t_2}^{k_{22}}...{t_r}^{k_{2r}}a^{l_2},...,{t_r}^{k_{rr}}a^{l_r},a^m\rangle\ \ \ \ \ \ (*)
\]for $k_{ii}>0$, $k_{ij}\geq 0$, $l_i \in \Z$ and $m>0$ an integer such that $\gcd(m,n)=1$ and $H \cap \left< a \right>= \left< a^m \right>$. Then every homomorphism $\xi:H \to \R$ can be extended to a homomorphism $\chi:\Gamma_n \to \R$.
\end{lema}

\begin{proof}
The equation $\chi|_H=\xi$ is equivalent to a system of $r$ equations
\[
\begin{cases}
\chi({t_1}^{k_{11}}...{t_r}^{k_{1r}}a^{l_1})=\xi({t_1}^{k_{11}}...{t_r}^{k_{1r}}a^{l_1}),\\
\chi({t_2}^{k_{22}}...{t_r}^{k_{2r}}a^{l_2})=\xi({t_2}^{k_{22}}...{t_r}^{k_{2r}}a^{l_2}),\\
\ \ \ \ \ \ \ \ \ \ \ \ \ \ \ \ \ \ \ \ \ \ \ \vdots \\
\ \ \ \ \ \ \ \ \chi({t_r}^{k_{rr}}a^{l_r})=\xi({t_r}^{k_{rr}}a^{l_r}).\\
\end{cases}
\]So, to create such an extension $\chi$ we just have to define $\chi(a)=0$ and define the real numbers $\chi({t_i})$ satisfying equations $(1)$ to $(r)$ above. Equation $(r)$ is equivalent to
\[
k_{rr}\chi(t_r)=\xi(t_r^{k_{rr}}a^{l_r}),
\]so if we define $\chi(t_r)=\frac{1}{k_{rr}}\xi(t_r^{k_{rr}}a^{l_r})$, equation $(r)$ is satisfied. Similarly, equation $(r-1)$ is equivalent to
\[
k_{r-1,r-1}\chi(t_{r-1})+k_{r-1,r}\chi(t_r)=\xi(t_{r-1}^{k_{r-1,r-1}}t_r^{k_{r-1,r}}a^{l_{r-1}}),
\]so if we define $\chi(t_{r-1})=\frac{1}{k_{r-1,r-1}}\xi(t_{r-1}^{k_{r-1,r-1}}t_r^{k_{r-1,r}}a^{l_{r-1}})-\frac{k_{r-1,r}}{k_{r-1,r-1}}\chi(t_r)$, equation $(r-1)$ is satisfied. By doing this recursively to all $i$, we are done.
\end{proof}

\begin{theorem}\label{meusigmah}
Let $H$ be a finite index subgroup of $\Gamma_n$, say,
\[
H=\langle {t_1}^{k_{11}}...{t_r}^{k_{1r}}a^{l_1},{t_2}^{k_{22}}...{t_r}^{k_{2r}}a^{l_2},...,{t_r}^{k_{rr}}a^{l_r},a^m\rangle\ \ \ \ \ \ (*)
\]for $k_{ii}>0$, $k_{ij}\geq 0$, $l_i \in \Z$ and $m>0$ an integer such that $\gcd(m,n)=1$ and $H \cap \left< a \right>= \left< a^m \right>$, and let $\alpha=a^m$ and $x_i={t_i}^{k_{ii}}...{t_r}^{k_{ir}}a^{l_i}$ be its generators. Then $\Sigma^1(H)^c=\{[\xi_1],...,[\xi_r]\}$, where $\xi_i(x_j)=k_{ji}$ if $j \leq i$ and $\xi_i(x_j)=0$ if $j>i$. 

In other words, if we identify $S(H) \simeq S^{r-1}$ as we did above, then
\[
\Sigma^1(H)^c= \Bigg\{ \frac{(k_{11},0,0,...,0)}{\Vert (k_{11},0,0,...,0) \Vert},\frac{(k_{12},k_{22},0,...,0)}{\Vert (k_{12},k_{22},0,...,0) \Vert},...,\frac{(k_{1r},k_{2r},k_{3r},...,k_{rr})}{\Vert (k_{1r},k_{2r},k_{3r},...,k_{rr}) \Vert}\Bigg\}.
\]
\end{theorem}

\begin{proof}
By Lemma \ref{extensionlemma}, $\Sigma^1(H)^c=i^*(\Sigma^1(\Gamma_n)^c)$ so by Theorem \ref{meusigmagamman}, $\Sigma^1(H)^c=\{[\chi_1|_H],...,[\chi_r|_H]\}$. Using that $\chi_i(t_j)=1$ if $i=j$ and $\chi_i(t_j)=0$, it is easy to see that the image of $[\chi_i|_H]$ (which we denote by $[\xi_i]$) under the homeomorphism $S(H) \simeq S^{r-1}$ described above is $\frac{(k_{1i},...,k_{ii},0,...,0)}{\Vert (k_{1i},...,k_{ii},0,...,0) \Vert}$. This completes the proof.
\end{proof}

\subsection{Finite index subgroups that are not $\Gamma_k$}

In \cite{Bogopolski} it was shown that every finite index subgroup of a solvable Baumslag-Solitar group $BS(1,n)$ is also (isomorphic to) a solvable Baumslag-Solitar group $BS(1,n^k)$ for some $k \geq 1$. Since the groups $\Gamma_n$ are generalizations of $BS(1,n)$, it is natural to ask whether
every finite index subgroup of $\Gamma_n$ is also (isomorphic to) another $\Gamma_k  \text{~for some~}  k \geq 2$.
In this section we show that this question has a negative answer. Below, we consider a specific class of finite index subgroups $H$ of $\Gamma_n$ for which we give necessary and sufficient conditions for $H$ to be isomorphic to $\Gamma_k$ for some $k\ge 2$. 

\begin{theorem}\label{hnotgamman}
Let $H$ be a finite index subgroup of $\Gamma_n$ such that
\[
H=\langle t_1^{k_{11}}t_2^{k_{12}}...t_r^{k_{1r}},t_2^{k_{22}}...t_r^{k_{2r}},...,t_r^{k_{rr}},a^m \rangle
\] with $k_{11}>0$, $0 \leq k_{ij}<k_{ii}$ for all $1 \leq i<j \leq r$ and $m>0$ such that $\gcd(m,n)=1$. Then
\[
H \simeq \Gamma_k\ \text{for some}\ k \geq 2 \text{~if and only if~} k_{ij}=0\ \text{for all}\ 1 \leq i<j \leq r.
\]
\end{theorem}

\begin{proof}
Suppose first that $k_{ij}=0$ for all $1 \leq i<j \leq r$. Then from Theorem \ref{presforh} we immediately get that $H \simeq \Gamma_k$ for $k=p_1^{y_1k_{11}}...p_r^{y_rk_{rr}}$. Suppose now that $H \simeq \Gamma_k$ for some $k \geq 2$ and write $k=q_1^{z_1}...q_s^{z_s}$, $q_1<q_2<...<q_s$, $z_i \geq 1$ the prime decomposition of $k$. Then in particular $s=card(\Sigma^1(\Gamma_k)^c)=card(\Sigma^1(H)^c)=r$, so $k=q_1^{z_1}...q_r^{z_r}$. By Theorem \ref{presforh}, $H$ has the presentation
\[
H=\langle \alpha,x_1,....,x_r\ |\ x_i\alpha x_i^{-1}=\alpha^{n_i},\ x_ix_j=x_jx_i\ \text{for all}\ i,j \rangle,
\]where $n_i=p_i^{y_ik_{ii}}...p_r^{y_rk_{ir}}$. There is also a split exact sequence
\[
1 \to ker(\pi) \to H \stackrel{\pi}{\to} {\Z}^r \to 1
\]where $\pi(x_i)=e_i$, $\pi(\alpha)=0$ and $ker(\pi)$ abelian. In particular, every element of $H$ can be written as $x_1^{\lambda_1}...x_r^{\lambda_r}u$ for some $\lambda_i \in \Z$ and $u \in ker(\pi)$. Since $H \simeq \Gamma_k$, then there must be $r+1$ elements inside $H$ (which are the images of the analogous $r+1$ elements in $\Gamma_k$), say, $X_i=x_1^{k_{i1}'}...x_r^{k_{ir}'}u_i$, $1 \leq i \leq r$ and $A=x_1^{\tilde{k_1}}...x_r^{\tilde{k_r}}\tilde{u}$ for some $k_{ij}',\tilde{k_i} \in \Z$ and $u_i,\tilde{u} \in ker(\pi)$, such that $
H=\langle X_1,...,X_r,A \rangle$ and $X_iAX_i^{-1}=A^{q_i^{z_i}}\ \text{for all}\ 1 \leq i \leq r$. By projecting any of these equations on $\Z^r$ we obtain $\tilde{k_1}=...=\tilde{k_r}=0$ and so $A=\tilde{u}=x_1^{-\lambda_1}...x_r^{-\lambda_r}\alpha^M x_r^{\lambda_r}...x_1^{\lambda_1}$ for some $\lambda_i \geq 0$ and $M \neq 0$. By replacing this in the $r$ equations above and using that $ker(\pi)$ is abelian and the $x_i$'s commute with each other, we obtain the $r$ equations in $H$
\begin{equation}
\label{(i)}
x_1^{k_{i1}'}...x_r^{k_{ir}'}\alpha^M x_r^{-k_{ir}'}...x_1^{-k_{i1}'}=\alpha^{Mq_i^{z_i}}
\end{equation}
for each $1 \leq i \leq r$. If a power $k_{ij}'$ is nonnegative we can use a relation of $H$ to conjugate $\alpha^M$. If it is negative, though, then since all the $x_i$ commute we can push the two $x_j$ from the left side to the right side of equation \eqref{(i)} and use the (now positive) power $-k_{ij}'$ to conjugate $\alpha^{Mq_i^{z_i}}$. Thus equation \eqref{(i)} will always imply an equality of a power of $\alpha^M$ with a power of $\alpha^{Mq_i^{z_i}}$. Since $H$ is torsion-free and $M \neq 0$, this yields an equation of prime decomposition which depends on the sign of the $k_{ij}'$. After a careful analysis of the possible prime decomposition equations we can conclude that $k_{ij}'$ is $1$ if $i=j$ and $0$ otherwise. The equations \eqref{(i)} become $x_i\alpha^Mx_i^{-1}=\alpha^{Mp_i^{z_i}}$. This implies $p_i^{y_ik_{ii}}p_{i+1}^{y_{i+1}k_{i,i+1}}...p_r^{y_rk_{ir}}=p_i^{z_i}$,
which implies $k_{i,i+1}=...=k_{ir}=0$. Since $i$ is arbitrary, we have that $k_{ij}=0$ for any $1 \leq i<j \leq r$, as desired.
\end{proof}

\section{Convex polytopes and property $R_\infty$}

In this section we show that finding a special kind of invariant convex polytope in the character sphere $S(G)$ is enough to guarantee property $R_\infty$ for a finitely generated group $G$ (Theorem \ref{meuteopolytope}). We will use a slightly more general version of Theorem $3.3$ in \cite{DaciDess}, which we state below. The proof is the same given there, just by observing that the authors didn't use directly the definition of $\Sigma^1(G)^c$ but only the fact that it is invariant in $S(G)$ (that is, invariant under all permutations of the form $[\chi]\mapsto[\chi \circ \varphi]$ for $\varphi \in Aut(G)$).

\begin{theorem}\label{teodacidessgeral}
Let $G$ be a finitely generated group. Suppose there is a nonempty and finite subset $A \subset S(G)$ which is invariant in $S(G)$, consisting only of rational points and contained in an open hemisphere of $S(G)$. Then $G$ has property $R_\infty$. \qed
\end{theorem}

Let $G$ be a finitely generated group whose abelianized group $G^{ab}$ has free rank $n$. Consider the homeomorphism

\begin{align*}
  &\mathfrak h:S(G)\longrightarrow S^{n-1} \\
  &[\chi]\longmapsto \frac{(\chi(x_1),...,\chi(x_n))}{\Vert (\chi(x_1),...,\chi(x_n)) \Vert},\\
\end{align*} where the $x_i \in G$ are the free-abelian generators of $G^{ab}$. Given $\varphi \in Aut(G)$, we have the induced homeomorphism $\varphi^*:S(G) \to S(G)$ with $\varphi^*[\chi]=[\chi \circ \varphi]$. Let $\varphi^S:S^{n-1} \to S^{n-1}$ be the composition $\varphi^S=\mathfrak h \circ \varphi^* \circ \mathfrak h^{-1}$.

By the definition above, $K \subset S(G)$ is invariant in $S(G)$ if and only if $\mathfrak h(K)$ is invariant under $\varphi^S$ for all $\varphi \in Aut(G)$. From now on, we assume the standard definitions of convex subsets and convex hulls of euclidean spaces $\R^d$. For spherical objects, the definitions will be the following:

\begin{defi}
Let $A \subset S^n \subset \R^{n+1}$ and suppose $A$ is contained in an open hemisphere of $S^n$, say, $A \subset O(v)=\{x \in S^n\ |\ \langle x,v \rangle > 0\}$ for some $v \in S^n$. We say that $A$ is (spherically) convex if for any $a_1,a_2 \in A$, $\gamma_{a_1,a_2}(t)=\frac{(1-t)a_1+ta_2}{\Vert (1-t)a_1+ta_2 \Vert} \in A$ for all $t \in [0,1]$. The convex hull of any subset $A \subset O(v)$ is the smallest convex subset of $O(v)$ which contains $A$ and is denoted by $conv(A)$.
\end{defi}

It is an easy task to show that $conv(A)$ above can be described as
\[
conv(A)=\left\{\frac{t_1a_1+...+t_ma_m}{\Vert t_1a_1+...+t_ma_m \Vert}\ |\ m \geq 1, a_i \in A, t_i>0 \right\}.
\] The following lemma shows a special property of the homeomorphisms $\varphi^S$.

\begin{lema}\label{meulema4}
The homeomorphism $\varphi^S: S^{n-1} \to S^{n-1}$ maps convex hulls to convex hulls. Precisely, let $A \subset O(v)$ and suppose $\varphi^S(A) \subset O(w)$ for some $w$. Then $\varphi^S(conv(A)) = conv(\varphi^S(A))$.
\end{lema}

\begin{proof}
Since ${(\varphi^{-1})}^S={(\varphi^S)}^{-1}$, it is enough to show that $\varphi^S(conv(A)) \subset conv(\varphi^S(A))$. Let $P \in conv(A)$ and write $P=\frac{t_1a_1+...+t_ma_m}{\Vert t_1a_1+...+t_ma_m \Vert}$ for some $a_i \in A$ and $t_i>0$. For each $a_i$, since $\mathfrak h:S(G) \to S^{n-1}$ is surjective we write $a_i=\mathfrak h[\chi_i]$ and by multiplying the representative $\chi_i$ by some $r>0$ if necessary we can actually suppose $a_i=\mathfrak h[\chi_i]=(\chi_i(x_1),...,\chi_i(x_n))$. Then, by definition, $\varphi^S(a_i)=\frac{1}{\lambda_i}(\chi_i \circ \varphi(x_1),...,\chi_i \circ \varphi(x_n))$,
where $\lambda_i=\Vert (\chi_i \circ \varphi(x_1),...,\chi_i \circ \varphi(x_n)) \Vert >0$. Now we compute $\varphi^S(P)$. It is easy to see that $\mathfrak h[t_1\chi_1+...+t_m\chi_m]=P$, since $a_i=\mathfrak h[\chi_i]$.
By denoting
\[
\lambda=\Vert (t_1(\chi_1 \circ \varphi)(x_1)+...+t_m(\chi_m \circ \varphi)(x_1),...,t_1(\chi_1 \circ \varphi)(x_n)+...+t_m(\chi_m \circ \varphi)(x_n)) \Vert,
\] we have

\begin{eqnarray*}
\varphi^S(P) 	&=& \frac{t_1}{\lambda}((\chi_1 \circ \varphi)(x_1),...,(\chi_1 \circ \varphi)(x_n))+...+\frac{t_m}{\lambda}((\chi_m \circ \varphi)(x_1),...,(\chi_m \circ \varphi)(x_n))\\
				&=& \frac{\lambda_1t_1}{\lambda}\varphi^S(a_1)+...+\frac{\lambda_mt_m}{\lambda}\varphi^S(a_m)\\
				&=& \frac{\frac{\lambda_1t_1}{\lambda}\varphi^S(a_1)+...+\frac{\lambda_mt_m}{\lambda}\varphi^S(a_m)}{\Vert \frac{\lambda_1t_1}{\lambda}\varphi^S(a_1)+...+\frac{\lambda_mt_m}{\lambda}\varphi^S(a_m) \Vert} \ \ \text{(since the above vector is already unitary)}\\
				&\in & conv(\varphi^S(A)),
\end{eqnarray*} as desired.
\end{proof}

Given an open hemisphere $O(v)=\{x \in S^n\ |\ \langle x,v \rangle > 0\}$ of $S^n$ for some $v \in S^n$, consider the affine $n$-space $v+\{v\}^\perp=\{v+w\ |\ \langle w,v \rangle =0\} \subset \R^{n+1}$. One can show that there is a homeomorphism $\theta_v:v+\{v\}^\perp \to O(v)\ \text{with}\ \theta_v(P)=\frac{P}{\Vert P \Vert}$, the inverse map given by $P \mapsto \frac{\Vert v \Vert^2}{\langle P,v \rangle}P$ (see next figure). From now on we identify $\R^n = v+\{v\}^\perp$.

\begin{figure}[htb!]
\centering
\includegraphics[scale=0.5]{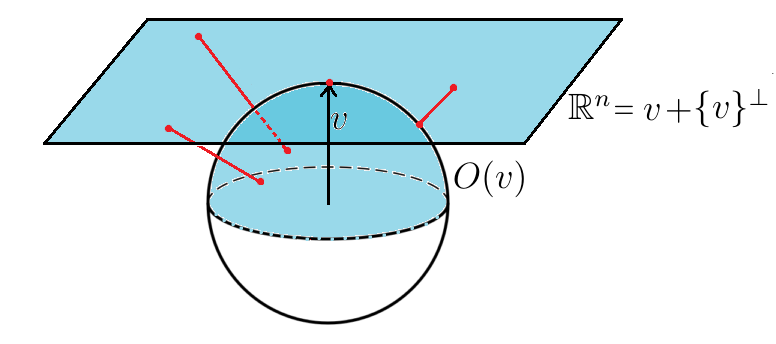}
\end{figure}

It is straightforward to show that $\theta_v:\R^n \to O(v)$ maps convex hulls of $\R^n$ to convex hulls of $O(v)$. Now we will define the convex polytopes in our context.

\begin{defi}[Euclidean convex polytopes]
A closed halfspace in $\R^d$ is a set of the form $H=\{x \in \R^d\ | \langle x,v \rangle \geq \beta\}$ for some $0 \neq v \in \R^d$ and $\beta \in \R$. A convex polytope $K$ in $\R^d$ is a finite intersection $K=\cap_{i=1}^n H_i$ of closed halfspaces $H_i$ which is also a bounded subset. Thinking of $K$ as a submanifold of $\R^d$ (with boundary), there is a well defined dimension $r=dim(K)$, so we say that $K$ is an $r$-polytope.
\end{defi}

We can always suppose that the family $\{H_i\}$ of closed halfspaces defining $K$ is irredundant, that is, is the minimal family necessary to define $K$.

\begin{defi}[Spherical convex polytopes]
For any $n \geq 0$, a closed hemisphere in $S^n$ is a set having the form $C(w)=\{p \in S^n\ |\ \langle p,w \rangle \geq 0\}$ for some $w \in S^n$. A convex polytope $K \subset S^n$ is a finite intersection of closed hemispheres in $S^n$. Given a finitely generated group $G$ with $S(G) \stackrel{\mathfrak h}{\simeq} S^{n-1}$, we say that $K \subset S(G)$ is a convex polytope if $\mathfrak h(K)$ is a convex polytope in $S^{n-1}$.
\end{defi}

The next lemma uses some known facts about Euclidean polytopes with which we will assume the reader is familiar.

\begin{lema}\label{minhaprop1}
Let $K \subset \R^d$ be a (Euclidean) $d$-polytope (maximal dimension) and $f:K \to K$ a homeomorphism. If $f$ maps segments to segments, that is, for any $P,Q \in K$, $f(conv(P,Q))=conv(f(P),f(Q))$, then $f$ maps vertices to vertices.
\end{lema}

\begin{proof}
Let $K=\cap_{i=1}^n H_i$ for an irredundant family $\{H_i\}$ and let $F_i=K \cap H_i$ be its facets. It is known that $n \geq d+1$, that $\partial K=F_1 \cup ... \cup F_n$ and that a point of $K$ is a vertex if and only if it belongs to at least $d$ different facets. Since $f$ is a homeomorphism, it must map the boundary $\partial K$ to itself, and so $f(F_1 \cup ... \cup F_n)=F_1 \cup ... \cup F_n$. Suppose by contradiction that a vertex $P \in K$ is mapped to a non-vertex point $f(P) \in K$ (but obviously $P,f(P) \in \partial K$). If a point $Q \in K$ belongs to any facet of $K$ containing $P$ (say, $F$), then $conv(Q,P) \subset F$, since every facet is convex. Then $conv(f(Q),f(P)) \subset f(F) \subset \partial K$ by hypothesis, so the whole straight path joining $f(Q)$ and $f(P)$ is contained in the boundary $\partial K$. Then one can show that $f(Q)$ must be in a facet which also contains $f(P)$. This argument shows that all the facets containing $P$ must be mapped into the facets containing $f(P)$. But there are at least $d$ facets containing $P$, say, $F_1,...,F_d$ and at most $d-1$ facets containing $f(P)$, say, $F_{i_1},...,F_{i_{d-1}}$. Then
\[
f(F_1 \cup ... \cup F_d) \subset F_{i_1} \cup ... \cup F_{i_{d-1}}.
\] We continue: since there are at least $d+1$ facets, let $Z \in \partial K$ be a point outside $F_{i_1} \cup ... \cup F_{i_{d-1}}$, say, $Z \in F_{i_d}$, and we can suppose $F_{i_d}$ is the only facet containing $Z$. Since $f$ is surjective, $Z=f(W)$, so $W$ must be a boundary point outside $F_1 \cup ... \cup F_d$, say, $W \in F_{d+1}$. By the same argument above, we must have $f(F_{d+1}) \subset F_{i_d}$ and so $f(F_1 \cup ... \cup F_{d+1}) \subset F_{i_1} \cup ... \cup F_{i_d}$. If $d+1=n$, we stop. If not, we follow these same steps. After a finite number of steps we will have
\[
f(F_1 \cup ... \cup F_n) \subset F_{i_1} \cup ... \cup F_{i_{n-1}},
\]so $f(\partial K) \subsetneq \partial K$, contradiction.
\end{proof}

\begin{theorem}
Let $G$ be a finitely generated group and $K \subset S(G)$ a convex polytope contained in an open hemisphere of $S(G)$. Then $K$ is invariant in $S(G)$ if and only if $V(K)$ is invariant in $S(G)$.
\end{theorem}

\begin{proof}
The convex polytope $\mathfrak h(K)$ is contained in some open hemisphere $O(v)$ of $S^{n-1}$. Let $\theta_v:\R^{n-1}\to O(v)$ be the homeomorphism previously defined. One can verify from the definition of $\theta_v$ that the preimage of a closed hemisphere in $S^{n-1}$ under $\theta_v$ is a closed halfspace in $\R^{n-1}$. Then to see that the preimage $K'={\theta_v}^{-1}(\mathfrak h(K))$ is a polytope it suffices to see that it is bounded. Since $\mathfrak h(K)$ is closed in the compact $S^{n-1}$, it is compact. Since $\theta_v$ is a homeomorphism, $K'$ is also compact in $\R^{n-1}$ and therefore bounded, so it is in fact a $r$-polytope for some $0 \leq r \leq n-1$.

To show the theorem, let $\varphi \in Aut(G)$. It is enough to show that $\mathfrak h(K)$ is invariant under $\varphi^S$ if and only if $V(\mathfrak h(K))$ is. Suppose first that $V(\mathfrak h(K))$ is invariant under $\varphi^S$. In Euclidean space, every convex polytope is the convex hull of its vertices. Since $\theta_v$ maps convex hulls to spherical convex hulls, it follows that $\mathfrak h(K)$ is also the convex hull of its vertices. Using Lemma \ref{meulema4}, we have
\[
\varphi^S(\mathfrak h(K))=\varphi^S(conv(V(\mathfrak h(K))))=conv(\varphi^S(V(\mathfrak h(K))))=conv(V(\mathfrak h(K)))=\mathfrak h(K),
\]as desired. Now, suppose $\varphi^S(\mathfrak h(K))=\mathfrak h(K)$. If $r<n-1$, then $K'$ is contained in a proper $r$-hyperspace of $\R^{n-1}$, say, $E^r$. There is a linear isomorphism and isometry $T:\R^r \to E^r$ and a $r$-polytope $\tilde{K}\subset \R^r$ such that $K'=T(\tilde{K})$. Consider the composition of homeomorphisms
\[
\tilde{K} \stackrel{T}{\longrightarrow} K' \stackrel{\theta_v}{\longrightarrow} \mathfrak h(K) \stackrel{\varphi^S}{\longrightarrow} \mathfrak h(K) \stackrel{{\theta_v}^{-1}}{\longrightarrow} K' \stackrel{T^{-1}}{\longrightarrow} \tilde{K}.
\]Since $T$ maps straight paths to straight paths, $\theta_v$ maps straight paths to geodesic paths and $\varphi^S$ maps geodesic paths to geodesic paths, this composition is a homeomorphism which maps straight paths to straight paths. Since $\tilde{K}$ has maximal dimension in $\R^r$, by Lemma \ref{minhaprop1} this composition must map the vertices of $\tilde{K}$ to themselves. Since the vertices of $\mathfrak h(K)$ are the image of the ones from $K'$, it follows that $\varphi^S$ must map the vertices of $\mathfrak h(K)$ to themselves, as desired. If $K'$ already had maximal dimension $r=n-1$, the proof is the same, but we don't even need to use $\tilde{K}$ and $T$.
\end{proof}

\begin{theorem}\label{meuteopolytope}
Let $G$ be a finitely generated group. If there is a convex polytope $K \subset S(G)$ contained in an open hemisphere of $S(G)$ and with rational vertices such that it is invariant under all homeomorphisms induced by automorphisms of $G$, then $G$ has property $R_\infty$. In particular, if $\Sigma^1(G)^c$ is one such polytope, then $G$ has property $R_\infty$.
\end{theorem}

\begin{proof}
By the previous theorem, $V(K) \subset S(G)$ is finite, invariant and by definition contained in an open half-space of $S(G)$. Then the result follows directly from Theorem \ref{teodacidessgeral}.
\end{proof}

\section{Property $R_\infty$ for $\Gamma_n$, its finite index subgroups, and direct products}

In this section we use all the information previously gathered to guarantee property $R_\infty$ for $\Gamma_n$ (Corollary \ref{meurinftygamman}), its finite index subgroups $H$ (Corollary \ref{rinftyforh}) and also for any (finite) direct product involving these groups (Corollary \ref{meurinfproddir}). Note that property $R_\infty$ is already known for $\Gamma_n$ and its finite index subgroups (see \cite{TabackWongGamman}). However, by using sigma theory, we obtain the same results with new and easier proofs. Corollary \ref{meurinfproddir} for the direct product was not considered in \cite{TabackWongGamman}. In Proposition \ref{example}, we exhibit a group $G$ where Theorem \ref{meuteopolytope} can be used to guarantee property $R_\infty$ without the need of completely computing the $\Sigma^1$ invariant.

We will make use of the following theorem.

\begin{theorem}[\cite{DaciDess}, Theorem $3.3$]\label{dacidess}
Let $G$ be a finitely generated group such that
\[
\Sigma^1(G)^c=\{[\chi_1],...,[\chi_m]\}
\]is a (nonempty) finite set of rational points. If $\{[\chi_1],...,[\chi_m]\}$ is contained in an open hemisphere of $S(G)$, then $G$ has property $R_\infty$.
\end{theorem}

\begin{cor}\label{meurinftygamman}
The generalized solvable Baumslag-Solitar groups $\Gamma_n$ have property $R_\infty$.
\end{cor}

\begin{proof}
Observe that, by Theorem \ref{meusigmagamman}, $\Sigma^1(\Gamma_n)^c$ is a finite set of rational points and is contained in the open hemisphere $O\left( \frac{(1,1,...,1)}{\Vert(1,1,...,1)\Vert} \right)$. The result follows from Theorem \ref{dacidess}.
\end{proof}

\begin{cor}\label{rinftyforh}
All finite index subgroups of $\Gamma_n$ have property $R_\infty$.
\end{cor}

\begin{proof}
Let $H$ be such finite index subgroup. As above, just observe that, by Theorem \ref{meusigmah}, $\Sigma^1(H)^c$ is a finite set of rational points and is contained in the open hemisphere $O\left( \frac{(1,1,...,1)}{\Vert(1,1,...,1)\Vert} \right)$ of $S(H)$. The result follows from Theorem \ref{dacidess}.
\end{proof}

Now we show property $R_\infty$ for any (finite) direct product between the groups $\Gamma_n$ and its finite index subgroups.

\begin{cor}\label{meurinfproddir}
Let $G=G_1 \times ... \times G_m$, where each $G_i$ is some $\Gamma_n$ or some finite index subgroup $H$ of $\Gamma_n$. Then $G$ has $R_\infty$ property.
\end{cor}

\begin{proof}
By Theorems \ref{meusigmagamman}, \ref{meusigmah} and by the known formula for the $\Sigma^1$ invariant of a direct product of groups (Proposition $A2.7$ of \cite{Strebel}, for example), we easily see that $\Sigma^1(G)^c$ is a finite set of rational points of $S(G)$. Furthermore, by Theorems \ref{meusigmagamman} and \ref{meusigmah}, we know that $\Sigma^1(G_i)^c$ is contained in an open hemisphere $O(v_i)$ of $S(G_i)$, for every $i$. From that, it is easy to see that $\Sigma^1(G)^c$ is contained in the open hemisphere $O(v_1,...,v_m)$ of $S(G)$. The result follows from Theorem \ref{dacidess}.
\end{proof}

Let $G$ be a finitely generated group and $X$ a finite set of generators for $G$. A path in the Cayley graph $\Gamma=\Gamma(G,X)$ of $G$ is denoted by $p=(g,y_1...y_n)$. The path $p$ starts at $g$, walks through the edge $(g,y_1)$ until the vertex $gy_1$, walks through $(gy_1,y_2)$ until $gy_1y_2$ and so on, until its terminus $gy_1...y_n$. Given $\chi \in Hom(G,\mathbb{R})$, the evaluation function $\nu_\chi$ is given by
\[
\nu_\chi(p)=\min\{\chi(g),\chi(gy_1),...,\chi(gy_1...y_n)\}.
\]
We are going to use the following geometric $\Sigma^1$-criterion given by R. Strebel (Theorem $A3.1$) in \cite{Strebel} in Proposition \ref{example} to illustrate a situation where we can use Theorem \ref{meuteopolytope} to guarantee property $R_\infty$ for a finitely generated group $G$ without having to completely compute $\Sigma^1(G)$. 

\begin{theorem}[Geometric Criterion for $\Sigma^1$]\label{geometriccrit}
Let $G$ be a finitely generated group with finite generating set $X$ and denote $Y=X^\pm$. Let $[\chi] \in S(G)$ and choose $t \in Y$ such that $\chi(t)>0$. Then the following are equivalent:
\begin{itemize}
\item[1)]$\Gamma_\chi$ is connected (or $[\chi] \in \Sigma^1(G)$);\\
\item[2)]For every $y \in Y$, there exists a path $p_y$ from $t$ to $yt$ in $\Gamma$ such that $\nu_\chi(p_y)>\nu_\chi((1,y))$.
\end{itemize}
\end{theorem}

\begin{prop}\label{example}
Let
\[
G=\langle a,t,s\ |\ tat^{-1}=a^n,\ sas^{-1}=a^m,\ tst^{-1}s^{-1}=a^r \rangle
\] for some coprime numbers $n,m\geq 2$ and some $r \in \Z$. Then $G$ has property $R_\infty$.
\end{prop}

\begin{proof}
We have the homeomorphism $\mathfrak h : S(G) \to S^1$, sending $[\chi]$ to the normalized of $(\chi(t),\chi(s))$. Let us compute $\Sigma^1(G)$ by the geometric criterion. Fix $X=\{a,t,s\}$ and $Y=\{a,a^{-1},t,t^{-1},s,s^{-1}\}$.
\begin{itemize}

\item[1)]\text{if} $\chi(t)<0$ \text{then} $[\chi] \in \Sigma^1(G)$. Fix $t^{-1}$ such that $\chi(t^{-1})>0$. By using the relations on $G$, one can see that the paths $p_a=(t^{-1},a^n)$, $p_{a^{-1}}=(t^{-1},a^{-n})$, $p_t=(t^{-1},t)$, $p_{t^{-1}}=(t^{-1},t^{-1})$, $p_s=(t^{-1},a^rs)$ and $p_{s^{-1}}=(t^{-1},s^{-1}a^{-r})$ satisfy 2) of \ref{geometriccrit}, so $[\chi] \in \Sigma^1(G)$.

\item[2)]\text{if} $\chi(s)<0$ \text{then} $[\chi] \in \Sigma^1(G)$. Similar to item 1).

\item[3)]\text{if} $\chi(t)=1$ \text{and} $\chi(s)=0$ \text{then} $[\chi] \notin \Sigma^1(G)$.

Suppose by contradiction that $[\chi] \in \Sigma^1(G)$. Then, in particular, there is a path $p=(1,w)$ in $\Gamma_\chi$ from $1$ to $t^{-1}at$. Write
\[
w=t^{k_{11}}s^{k_{12}}a^{r_1}...t^{k_{c1}}s^{k_{c2}}a^{r_c}.
\]Since $p$ is contained in $\Gamma_\chi$, $\chi(t)=1$ and $\chi(s)=0$ we must have
\[
k_{11} \geq 0,\ k_{11}+k_{21} \geq 0,\ ...,\ k_{11}+...+k_{c-1,1} \geq 0\ \text{and}\ k_{11}+...+k_{c1}=0.
\]By using the relations on $G$, we push right $t^{k_{11}}$ until $t^{k_{21}}$, then we push right $t^{k_{11}+k_{21}}$ until $t^{k_{31}}$, and so on. Since $k_{11}+...+k_{c1}=0$, we eliminate from $w$ all the $t$-letters and (after relabeling the $s$ and $a$ powers) we can write $w=s^{k_1}a^{r_1}...s^{k_c}a^{r_c}$ in $G$. But, as a vertex, $w$ must be the end of the path $p$. So we have $w=t^{-1}at$ and therefore
\[
a=twt^{-1}=t(s^{k_1}a^{r_1}...s^{k_c}a^{r_c})t^{-1}=(a^r s)^{k_1}a^{nr_1}...(a^rs)^{k_c}a^{nr_c},
\]or
\[
w'=(a^r s)^{k_1}a^{nr_1}...(a^rs)^{k_{c-1}}a^{nr_{c-1}}(a^rs)^{k_c}a^{nr_c -1}=1
\]in $G$. By projecting this equation onto the $s$-coordinate, we have $k_1+...+k_c=0$. Also, $(a^rs)a^M=a^{mM}(a^rs)$ and $a^M(a^rs)^{-1}=(a^rs)^{-1}a^{mM}$ for every $M \in \Z$. This means that, in $w'$, the entire positive pieces $(a^r s)^{k_i}$ can be pushed right and the negative ones can be pushed left. After doing this, we obtain an expression of the form
\[
(a^rs)^{-\lambda}a^{\alpha_1nr_1+...+\alpha_{c-1}nr_{c-1}+\alpha_c(nr_c -1)}(a^rs)^{\lambda}=1,
\]where each $\alpha_i$ is either $1$ or a positive power of $m$. This easily implies 
\[
\alpha_1nr_1+...+\alpha_{c-1}nr_{c-1}+\alpha_c(nr_c -1)=0.
\]
By putting all the multiples of $n$ above to the left and only $\alpha_c$ on the right, we get either $Mn=1$ (contradiction with the fact $n \geq 2$) or $Mn=m^Q$ for $Q \geq 1$ (contradiction with the fact $\gcd(n,m)=1$). This shows item $3)$.

\item[4)]\text{if} $\chi(t)=0$ \text{and} $\chi(s)=1$ \text{then} $[\chi] \notin \Sigma^1(G)$. Similar to item 3).
\end{itemize}

Now identify $S(G)$ with $S^1$ by the homeomorphism $\mathfrak h$ and let $[\chi_1]$ and $[\chi_2]$ be the points of items 3) and 4), respectively. Items $1)$ and $2)$ showed that the geodesic $\gamma$ in $S(G)$ between these points contains $\Sigma^1(G)^c$. We claim that $\gamma$ is invariant in $S(G)$. In fact, if $\varphi \in Aut(G)$ and $p \in \gamma$, then by Lemma \ref{meulema4} $\varphi^*(p)$ must be in the geodesic between $\varphi^*[\chi_1]$ and $\varphi^*[\chi_2]$. By the $\Sigma$ invariance and by items $3)$ and $4)$, $\varphi^*[\chi_1]$ and $\varphi^*[\chi_2]$ are in $\Sigma^1(G)^c$; therefore, by items $1)$ and $2)$, they must be in $\gamma$. Since $\gamma$ is a convex subset we have $\varphi^*(p) \in \gamma$, which shows our claim. Thus, in $S(G)$ we have $\gamma$ an invariant convex 1-dimensional polytope with the two rational vertices $[\chi_i]$ and the proposition follows from Theorem \ref{meuteopolytope}.
\end{proof}

\begin{remark}
In Proposition \ref{example}, if $r \neq 0$, we do not know whether the group $G$ is metabelian in general. While in such cases the proof of Theorem \ref{meusigmagamman} does not necessarily apply, the geometric criterion does apply. Of course, if $r=0$, we have $G=\Gamma(S)$ for $S=\{n,m\}$, so $G$ is metabelian. Therefore, Proposition \ref{example} illustrates an alternative way to derive property $R_{\infty}$ besides using the BNS invariant $\Sigma^1$.
\end{remark}

\end{document}